\documentclass[12pt]{article}
\usepackage{hyperref}
\usepackage{ifdraft}
\usepackage{amssymb,amsmath,amscd,epsfig,amsthm,verbatim,eqnarray}
\usepackage{fullpage,verbatim} %Makes text take up more of the page
\usepackage{authblk}
\usepackage{dsfont}

\newcommand{\Z}{\mathbb Z}
\newcommand{\Aut}{\mathop{\rm Aut}\nolimits}
\newcommand{\Aff}{\mathop{\rm Aff}\nolimits}
\newcommand{\SHAut}{\mathop{\rm SHAut}\nolimits}
\newcommand{\Sym}{\mathop{\rm Sym}\nolimits}
\newcommand{\bb}{\mathbf b}
\newcommand{\ba}{\mathbf a}
\newcommand{\bx}{\mathbf x}
\newcommand{\be}{\mathbf e}
\newcommand{\A}{\mathcal A}
\newcommand{\G}{\mathcal G}
\newcommand{\LL}{\mathcal L}
\newcommand{\one}{\mathds 1}
\newcommand{\zero}{\mathds O}

\newtheorem{theorem}{Theorem}[section]
\newtheorem{corollary}[theorem]{Corollary}
\newtheorem{proposition}[theorem]{Proposition}
\newtheorem{lemma}[theorem]{Lemma}

\theoremstyle{definition}
\newtheorem{definition}{Definition}

\newtheorem{remark}[theorem]{Remark}

\newtheorem*{theoremmain}{Theorem~\ref{thm:main}}
\newtheorem*{theoremaffine}{Theorem~\ref{thm:affine}}

\title{The lamplighter group of rank two generated by a bireversible automaton}
\author[1]{Elsayed Ahmed\footnote{Email: \href{mailto:eahmed1@mail.usf.edu}{eahmed1@mail.usf.edu}}}
\author[1]{Dmytro Savchuk\footnote{Email: \href{mailto:savchuk@usf.edu}{savchuk@usf.edu}; Partially supported by the Simons Collaboration Grant \#317198 from Simons Foundation}}
\affil[1]{Department of Mathematics and Statistics\\
               University of South Florida\\
               4202 E Fowler Ave\\
               Tampa, FL 33620-5700}

\begin{document}
\maketitle

\begin{abstract}
We construct a 4-state 2-letter bireversible automaton generating the lamplighter group $(\Z_2^2)\wr\Z$ of rank two. The action of the generators on the boundary of the tree can be induced by the affine transformations on the ring $\Z_2[[t]]$ of formal power series over $\Z_2$.\end{abstract}

\section{Introduction}
\label{sec:int}

The lamplighter group and its generalizations have been studied extensively during the last several decades. The simplest lamplighter group $\LL=\LL_{1,2}$ is defined as the restricted wreath product $\Z_2\wr\Z$ or, equivalently, as $\oplus_{\Z}\Z_2\rtimes\Z$, where $\Z$ acts on $\oplus_{\Z}\Z_2$ by shifting the index.
More generally, higher rank lamplighter groups $\LL_{n,d}$ are defined similarly as $(\Z_d^n)\wr\Z$.

All of the groups in this class are metabelian groups of exponential growth that, despite seemingly simple algebraic structure, possess extraordinary properties, connect different areas of mathematics, and serve as counterexamples to several conjectures. To start with, this is one of the simplest examples of finitely generated but not finitely presented groups.

Especially fruitful from historical perspective was the discovery that the lamplighter group (of rank 1) $\LL$ is one of the 6 nonisomorphic groups generated by 2-state 2-letter invertible Mealy automata~\cite{gns00:automata}. Every group generated by a $d$-letter automaton naturally acts on the regular $d$-ary tree, whose set of vertices can be identified with the set $X^*$ of all finite words over a $d$-letter alphabet $X=\{0,1,\ldots,d-1\}$. The automaton realization of $\LL$ has eventually led to the construction of counterexamples to various versions of Atiyah conjecture on $l^2$-Betti numbers~\cite{atiyah:elliptic} in~\cite{grigorch_lsz:atiyah,dicks_s:spectral_measure02,austin:irrational_dimension13,lehner_w:lamplighter_atiyah13,grabowski:atiyah_problem14} and more recently~\cite{grabowski:irrational_l2_lamplighter16}. The striking property of $\LL$ behind the first construction in~\cite{grigorch_lsz:atiyah} is that the spectrum of the Laplace operator on the Cayley graph of $\LL$ (with respect to a special system of generators) is pure
point spectrum~\cite{grigorch_z:lamplighter}, which was obtained via the action of $\LL$ on the boundary $X^\infty$ of the rooted binary tree $X^*$ induced by its automaton realization.

The lamplighter type groups play also an important role in the theory of random walks on groups~\cite{kaimanovich_v:random_walks83,lyons_pp:random_walks_lamplighter96,pittet_s:random_walks02,brieussel_tz:random_walks_affine_group}.
The subgroup structure of $\LL_{n,p}$ was explicitly described by Grigorchuk and Kravchenko in~\cite{grigorchuk_k:subgroup_structure14}. Further, the lamplighter groups happen to be one of the first examples of scale-invariant groups that are not virtually nilpotent~\cite{nekrashevych_p:scale_invariant}, thus giving a counterexample to a conjecture by Benjamini. Here by scale-invariant group we mean finitely generated infinite group $G$ that possess a nested sequence of finite index subgroups $G_n$ that are all isomorphic to $G$ and whose intersection is a finite group. And again, the automaton realization of lamplighter groups played an important role in this construction. In particular, it is proved in~\cite{grigorch_s:essfree} that each self-replicating automaton group acting essentially freely on the boundary of a tree (i.e., in such a way that for every nontrivial element of the group the measure of the fixed point set of this element under the action of the group on the boundary of the rooted tree is zero) is scale-invariant. Many lamplighter type groups happen to act essentially freely on the boundary.

Below we survey some of a history of generating lamplighter type groups by automata. But before proceeding, we first informally recall the classes of reversible and bireversible automata (the formal treatment is given in Section~\ref{sec:pre}). For every invertible $n$-state $m$-letter automaton $\A$ one can define an $m$-state $n$-letter \emph{dual} automaton $\partial\A$ by ``swapping the roles'' of letters of the alphabet and the states of an automaton. I.e., there is a transition $q\stackrel{x/y}{\longrightarrow}w$ in $\A$ for states $q,w$ of $\A$ and letters $x,y$ from the alphabet if and only if there is a transition $x\stackrel{q/w}{\longrightarrow}y$ in $\partial\A$. An automaton $\A$ is called \emph{reversible} if its dual $\partial\A$ is invertible, and it is called \emph{bireversible} if $\A$, $\partial A$ and $\partial(A^{-1})$ are all invertible. Bireversible automata constitute a very interesting class. The groups that they generate are often hard to analyze by the standard methods based on various contraction properties. For example, several families of bireversible automata studied in~\cite{gl_mo:compl,vorobets:aleshin,vorobets:series_free,steinberg_vv:series_free,savchuk_v:free_prods} generate free nonabelian groups or free products of finite number of copies of $\Z_2$. But the proof behind the main break through construction in~\cite{vorobets:aleshin} that answered 20 year old question by Sidki is rather technical and involved.

The main result of this paper is the following theorem.

\begin{theoremmain}
The group $\G=\langle a=(b,d)\sigma,b=(d,b)\sigma,c=(a,c),d=(c,a)\rangle$ is generated by a 4-state 2-letter bireversible automaton (depicted in Figure~\ref{fig:aut}) and is isomorphic to the rank $2$ lamplighter group $\LL_{2,2}\cong\Z_2^2\wr\Z$.
\end{theoremmain}

With the result above in mind, we proceed to surveying the history of the topic to motivate our interest in the group $\G$. After the original realization of $\LL_{1,2}$ in~\cite{gns00:automata,grigorch_z:lamplighter} by a 2-state 2-letter automaton, Silva and Steinberg in~\cite{silva_s:lamplighter05} have constructed a family of $n^d$-state $n^d$-letter reset automata generating $\LL_{n,d}$. Thus, the group $\LL_{2,2}$ in this family is generated by 4-state 4-letter automaton.

In~\cite{bartholdi_s:bsolitar} Bartholdi and \v Sun\'ic have constructed a large family of lamplighter groups $\LL_{n,p}$ parametrized by monic polynomials over $\Z_p$ that are invertible in the ring $\Z_p[[t]]$ of formal power series over $\Z_p$. There are two automata in this family generating $\LL_{2,2}$: these automata correspond to polynomials $t^2+t+1$
%$a=(d,b)\sigma, b=(d,b),c=(c,a),d=(c,a)\sigma$
and $t^2+1$.
%$a=(d,b), b=(d,b)\sigma,c=(c,a),d=(c,a)\sigma$
Both of these automata are 4-state 2-letter automata, but none of them is bireversible (however, their inverses are reversible). Another reincarnation of  $\LL_{2,2}$ was discovered in~\cite{grigorch_s:essfree}, where it was shown that a 3-state 2-letter automaton (not reversible) generates an index 2 extension of $\LL_{2,2}$.

The first examples of bireversible automata generating lamplighter type groups were constructed in~\cite{bondarenko_dr:lamplighterZ3wrZ16} and~\cite{savchuk_s:affine_automata}. Bondarenko, D'Angeli, and Rodaro in~\cite{bondarenko_dr:lamplighterZ3wrZ16} presented a 3-state 3-letter self-dual bireversible automaton generating $\LL_{1,3}\cong \Z_3\wr\Z$. The second automaton was a 4-state 2-letter automaton generating an index 2 extension of $\LL_{2,2}$. It appeared as the main example in the paper by Sidki and the second author~\cite{savchuk_s:affine_automata} whose main goal was to understand the action of lamplighter type groups on rooted trees.

In most cases when lamplighter group is generated by finite automaton acting on a binary tree, the base group $\oplus_{\Z}\Z_2$, which is generated by infinite number of commuting involutions, consists of spherically homogeneous automorphisms (an automorphism in called spherically homogeneous provided that for each level its states at the vertices of this level all coincide). Thus the whole group usually lies in the normalizer of the group $\SHAut(X^*)$ of all spherically homogeneous automorphisms inside the group $\Aut(X^*)$ of all automorphisms of $X^*$.
It was shown in~\cite{savchuk_s:affine_automata} that this normalizer consists of affine automorphisms of $X^*$ coming from the affine actions on the boundary $X^\infty$ of the tree viewed as uncountable infinite dimensional vector space over $\Z_2$. Moreover, it was shown that every realization of lamplighter group as an automaton group acting on the binary tree is conjugate to the one coming from affine actions. The group $\G$ is not an exception, but it turns out that elements of $\G$ sit in a more narrow class of automorphisms induced by the affine transformations of the boundary of $X^*$ viewed as the ring $\Z_2[[t]]$. For $f,g\in\Z_2[[t]]$ let $\tau_{f,g}$ denote the automorphism of the $X^*$ sending $h(t)$ to $h(t)\cdot f(t)+g(t)$. Then we can describe the generators of $\G$ as follows.

\begin{theoremaffine}
The generators $a,b,c$ and $d$ of $\G$ all lie in $\Aff_{[[t]]}(X^*)$ and are induced by the affine transformations of the form
\[
\begin{array}{lll}
a=\tau_{\frac{t^2+t+1}{t^2+1},\frac{1}{(t+1)^3}},&&b=\tau_{\frac{t^2+t+1}{t^2+1},\frac{t^2+t+1}{(t+1)^3}},\\ c=\tau_{\frac{t^2+t+1}{t^2+1},\frac{t}{(t+1)^3}},&&d=\tau_{\frac{t^2+t+1}{t^2+1},\frac{t^2}{(t+1)^3}}.
\end{array}
\]
\end{theoremaffine}

There is another motivation to study the group $\G$ from Theorem~\ref{thm:main}. It was initially one of the six groups among those generated by 7421 non-minimally symmetric 4-state invertible automata over 2-letter alphabet studied in~\cite{caponi:thesis2014}, for which the existence of elements of infinite order could not be easily established by the standard known methods. Note that very recently Gillibert~\cite{gillibert:order_undecidable18} has shown that the order problem is undecidable in the class of automaton groups, so there is no hope to have a unified algorithm working in all cases. Moreover, slightly later Bartholdi and Mitrofanov showed that, perhaps, quite surprisingly, the problem remains undecidable even in the class of contracting automaton groups~\cite{bartholdi_m:order_undecidable}.  In~\cite{klimann_ps:orbit_automata} many elements of infinite order in two of these six groups were found using a new technique of orbit automata. And the mentioned example from~\cite{savchuk_s:affine_automata} was one of these two groups. In this paper we use a similar approach to study the structure of the group $\G$, but our proof is somewhat simpler and the automaton that we study generates exactly $\LL_{2,2}$, and not its index 2 extension like in~\cite{savchuk_s:affine_automata}.

The paper has the following simple structure. In Section~\ref{sec:pre} we introduce basic notions and terminology on trees and automata. Section~\ref{sec:main} is devoted to the proof of the main result of the paper.

\section{Preliminaries}
\label{sec:pre}
We start this section by introducing the notions and terminology of automorphisms of regular rooted trees and Mealy automata. For more details, the reader is referred to~\cite{gns00:automata}.

Let $X=\{0,1,\ldots ,d-1\}$ be a finite set with $d\geq 2$ elements (called letters) and let $X^*$ denote the set of all finite words over the alphabet $X$. The set $X^*$ can be given the structure of a rooted $d$-ary tree by declaring that $v$ is adjacent to $vx$ for every $v\in X^*$ and $x\in X$. The empty word corresponds to the root of the tree and for each positive integer $n$ the set $X^n$ corresponds to the $n$-th level of the tree. Also the set $X^\infty$ of all infinite words over $X$ can be identified with the \emph{boundary} of the tree $X^*$ consisting of all infinite paths in the tree without backtracking initiating at the root. We will consider automorphisms of the tree $X^*$ (bijections of $X^*$ that preserve adjacency of vertices). The group of all automorphisms of $X^*$ is denoted by $\Aut(X^*)$. To describe such automorphisms, we will use the language of Mealy automata.

\begin{definition}
A \emph{Mealy automaton} (or simply \emph{automaton}) is a tuple $(Q,X,\pi,\lambda)$, where $Q$ is a set (the set of states), $X$ is a finite alphabet, $\pi\colon Q\times X\to Q$ is the \emph{transition function}
and $\lambda\colon Q\times X\to X$ is the \emph{output function}. If the set of states $Q$ is finite, the automaton is called \emph{finite}. If for every state $q\in Q$ the output function $\lambda_q(x)=\lambda(q,x)$ induces
a permutation of $X$, the automaton $\A$ is called \emph{invertible}. Selecting a state $q\in Q$ produces an \emph{initial automaton} $\A_q$.
\end{definition}

Automata are often represented by their \emph{Moore diagrams}. The Moore diagram of an automaton $\A=(Q,X,\pi,\lambda)$ is a directed graph in which the vertices are the states of $Q$ and the edges have the form $q\stackrel{x|\lambda(q,x)}{\longrightarrow}\pi(q,x)$ for
$q\in Q$ and $x\in X$. Figure~\ref{fig:aut} shows the Moore diagram of the automaton $\A$ that gives rise to our group $\G$ defined in Theorem~\ref{thm:main}.

Every invertible initial automaton $\A_q$ induces an automorphism of $X^*$, which will be also denoted by $\A_q$, defined as follows. Given a word
$v=x_1x_2x_3\ldots x_n\in X^*$, it scans its first letter $x_1$ and outputs $\lambda(q,x_1)$. The rest of the word is handled similarly by the initial automaton $\A_{\pi(q,x_1)}$. So we can actually extend the functions $\pi$ and $\lambda$ to $\pi\colon Q\times X^*\to Q$ and $\lambda\colon  Q\times X^*\to X^*$ via the equations
\[\begin{array}{l}
\pi(q,x_1x_2\ldots x_n)=\pi(\pi(q,x_1),x_2x_3\ldots x_n),\\
\lambda(q,x_1x_2\ldots x_n)=\lambda(q,x_1)\lambda(\pi(q,x_1),x_2x_3\ldots x_n).\\
\end{array}
\]

An automorphism of $X^*$ naturally induces an action on  the boundary of the tree, the set $X^\infty$. In fact, $X^\infty$ endowed with a natural topology in which two infinite words are close if they have large common prefix, is homeomorphic to the Cantor set and every automorphism of $X^*$ gives rise to a homeomorphism of $X^\infty$.

\begin{definition}
The semigroup (group) generated by all states of an automaton $\A$ viewed as automorphisms of the rooted tree $X^*$ under the operation of composition is called an \emph{automaton semigroup (group)} and denoted by $\mathbb{S}(\A)$ (respectively $\mathbb{G}(\A)$).
\end{definition}

In the definition of an automaton, we do not require  the set $Q$ of states to be finite. With this convention, the notion of an automaton group is equivalent to the notions of \emph{self-similar group}~\cite{nekrash:self-similar} and \emph{state-closed group}~\cite{nekrash_s:12endomorph}. However, most of the interesting examples of automaton groups are finitely generated groups defined by finite automata.

We now introduce the notion of a \emph{state} (also called \emph{section}) of an automorphism at a vertex of the tree. Let $g$ be an automorphism of the tree $X^*$ and $x\in X$. For any $v\in
X^*$ we can write \[g(xv)=g(x)v'\] for some $v'\in X^*$. Then the map $g|_x\colon X^*\to X^*$ given by \[g|_x(v)=v'\] defines an automorphism of $X^*$ which we call the \emph{state} of $g$ at vertex $x$. We can inductively extend the definition of a section to any finite word $x_1x_2\ldots x_n\in X^*$ as follows. \[g|_{x_1x_2\ldots x_n}=g|_{x_1}|_{x_2}\ldots|_{x_n}.\]

In fact, any automorphism of $X^*$ can be induced by an invertible initial automaton. Given an automorphism $g$ of $X^*$, we construct an invertible initial automaton $\A(g)$ whose action on $X^*$ coincides with that of $g$ as follows. The set of states of $\A(g)$ is the set $\{g|_v\colon  v\in X^*\}$
of different states of $g$ at the vertices of $X^*$. The transition and output functions are defined by
\[\begin{array}{l}
\pi(g|_v,x)=g|_{vx},\\
\lambda(g|_v,x)=g|_v(x).
\end{array}\]

We will adopt the following convention throughout the paper. If $g$ and $h$ are elements of some group acting on a set $Y$ and $y\in Y$, then $$gh(y)=h(g(y)).$$
Hence the states of any element of an automaton group can be computed as follows. If $g=g_1g_2\ldots g_n$ and $v\in X^*$, then $$g|_v=g_1|_v\cdot g_2|_{g_1(v)}\cdots
g_n|_{g_1g_2\cdots g_{n-1}(v)}.$$

For each automaton group $G$ there is a natural embedding
\[G\hookrightarrow G \wr \Sym(X)\] given by
\begin{equation}
\label{eq:wreath}
G\ni g\mapsto (g_1,g_2,\ldots,g_d)\sigma_g\in G\wr \Sym(X),
\end{equation}
where $g_1,g_2,\ldots,g_d$ are the states of $g$ at the vertices of the first level, and $\sigma_g$ is the permutation of $X$ induced by the action of $g$ on the first level of the tree. If $\sigma_g$ is the trivial permutation, it is customary to omit it in the notation. In the case of a binary rooted tree $\{0,1\}^*$, there is only one nontrivial permutation, namely the transposition $(01)$, which is usually denoted simply by $\sigma$.

The decomposition of all generators of an automaton group under the embedding~\eqref{eq:wreath} is called the \emph{wreath recursion} defining the group. It is a convenient language when doing computations involving the states of automorphisms. Indeed, products and inverses of automorphisms can be found as follows. If $g\mapsto (g_1,g_2,\ldots,g_d)\sigma_g$ and $h\mapsto (h_1,h_2,\ldots,h_d)\sigma_h$ are two elements of $\Aut(X^*)$, then $$gh=(g_1h_{\sigma_g(1)},g_2h_{\sigma_g(2)},\ldots,g_{d}h_{\sigma_g(d)})\sigma_g\sigma_h$$ and
$$g^{-1}=(g^{-1}_{\sigma_g^{-1}(1)},g^{-1}_{\sigma_g^{-1}(2)},\ldots,g^{-1}_{\sigma_g^{-1}(d)})\sigma_g^{-1}$$

We proceed now to the definition of bireversible automaton.

\begin{definition}
Given an automaton $\A=(Q,X,\pi,\lambda)$, we define its inverse as an automaton $\A^{-1}=(Q^{-1},X,\tilde{\pi},\tilde{\lambda})$, where
\[\begin{array}{l}
\tilde{\pi}(q^{-1},x)=\pi(q,\lambda_q^{-1}(x))^{-1},\\
\tilde{\lambda}(q^{-1},x)=\lambda(q,\lambda_q^{-1}(x)).
\end{array}\]
\end{definition}
For every finite automaton, we can construct the \emph{dual automaton} by switching the roles of states and the alphabet and switching the transition and output functions.
\begin{definition}
Given a finite automaton $\A=(Q,X,\pi,\lambda)$, its dual is a finite automaton $\partial\A=(X,Q,\hat{\lambda},\hat{\pi})$, where
\[\begin{array}{l}
\hat\lambda(x,q)=\lambda(q,x),\\
\hat\pi(x,q)=\pi(q,x).
\end{array}\]
for every $x\in X$ and $q\in Q$.
\end{definition}

It is easy to see that the dual of the dual of an automaton $\A$ coincides with $\A$. The semigroup $\mathbb{S}(\partial\A)$ generated by $\partial\A$ acts on the free monoid $Q^*$. This induces the action on $\mathbb{S}(\A)$. Similarly, $\mathbb{S}(\A)$ acts on $\mathbb{S}(\partial\A)$.

\begin{definition}
An automaton $\A$ is called \emph{reversible} if its dual is invertible. It is called \emph{bireversible} if it is invertible, and both $\A$ and $\A^{-1}$ are reversible.
\end{definition}

\begin{proposition}
The automaton $\A$ shown in Figure~\ref{fig:aut} generating the group $\G$ is bireversible.
\end{proposition}
\begin{proof}
By construction the automaton $\A$ is invertible and its dual $\partial\A$ is given by the wreath recursion
\[
\begin{array}{lcl}
\zero&=&(\one,\one,\zero,\zero)(abdc),\\
\one&=&(\zero,\zero,\one,\one)(ad).
\end{array}\]
So $\partial\A$ is invertible and $\A$ is reversible. Similarly one can check that the dual of $\A^{-1}$ is also invertible. Therefore, $\A$ is a bireversible automaton.
\end{proof}

In order to describe the lamplighter structure of $\G$ in Theorem~\ref{thm:main} we need to introduce the notions of spherically homogeneous~\cite{grigorch_s:essfree} and affine~\cite{savchuk_s:affine_automata,olijnyk_s:free_prods_unitriangular04} automorphisms of the tree.

\begin{definition}
An automorphism $g$ of the tree $X^*$ is called \emph{spherically homogeneous} if for each level $l$ the states of $g$ at the vertices of $X^l$ all coincide.
\end{definition}

Each such automorphism has a form $a=(b,b,\ldots,b)\sigma_1$, $b=(c,c,\ldots,c)\sigma_2,\ldots$, where $\sigma_i$'s are permutations of $X$. For example, automorphisms $a=(a,a)\sigma,b=(a,a)$ are spherically homogeneous automorphisms of the binary tree. It is not hard to see that the following conditions are equivalent:
\begin{itemize}
\item An automorphism $g$ of $X^*$ is spherically homogeneous.
\item For each level $l$ of the tree $X^*$ the states of $g$ at all the vertices of $X^l$ act identically on the first level.
\item For all words $u,v\in X^*$ of the same length, $g|_u=g|_v$.
\end{itemize}

Every spherically homogeneous automorphism can be fully defined by a sequence $\{\sigma_n\}_{n\geq 1}$ of permutations of $X$ where $\sigma_n$ describes the action of $g$ on the $n$-th letter of the input word. Given a sequence $\{\sigma_n\}_{n\geq 1}$, we can denote the corresponding spherically homogeneous automorphism by $[\sigma_n]_{n\geq 1}$.

The set of all spherically homogeneous automorphisms of $X^*$ forms a subgroup $\SHAut(X^*)$ of $\Aut(X^*)$. The description of spherically homogeneous automorphisms given above assures that $\SHAut(X^*)$ is isomorphic to an infinite product of countably many copies of $\Sym(|X|)$ and hence is uncountable. In the case of the binary tree, this group is abelian and isomorphic to the direct product of countably many copies of $\Z_2$. When $d\geq 3$, the group $\SHAut(X^*)$ contains an abelian subgroup consisting of automorphisms whose sections act on the first level as powers of some fixed long $d$-cycle.

In the case of a binary tree, there is a countably generated self-similar dense subgroup $\Delta$ in $\SHAut(X^*)$ defined below.

For an automorphism $g\in\Aut(X^*)$, we will denote by $g^{(n)}$ the automorphism of $X^*$ acting trivially on the $n$-th level, and whose states at all vertices of $X^n$ are equal to $g$. For example, $g^{(0)}=g$, $g^{(1)}=(g,g)$, etc. In particular, we will denote by $\sigma^{(n)}$, $n\geq 0$ the automorphism of $X^*$ that acts on the $(n+1)$-st coordinate in the input word by the long cycle $\sigma$. So $\sigma^{(0)}=(1,1)\sigma$, $\sigma^{(1)}=(\sigma^{(0)},\sigma^{(0)})$, $\ldots$ , $\sigma^{(n+1)}=(\sigma^{(n)},\sigma^{(n)})$. Now define a group $\Delta$ as
\[\Delta=\langle \sigma^{(0)},\sigma^{(1)},\sigma^{(2)},\ldots\rangle.\]

%For a finite state automorphism $g$ of $X^*$, it is algorithmically decidable whether $g$ is spherically homogeneous or not. First we check if all the states of $g$ at the vertices of the first level coincide. If not, then $g$ is not spherically homogeneous. Otherwise, we repeat the procedure for the state $g|_0$ (note that $g|_0=g|_1=\cdots=g|_{d-1}$). Since $g$ is finite state, this procedure will eventually terminate.

To conclude this section we discuss the notion of affine automorphisms and a related notion of $\Z_d[[t]]$-affine automorphisms. As described above, one of the ways to define an automorphism of $X^*$ is by using the language of automata. However, some automorphisms can be defined also differently. Namely, if we endow the boundary $X^{\infty}$ of the tree with some structure, then some of the transformations of $X^\infty$ will induce automorphisms of $X^*$. For example, one can endow $X^\infty$ with the structure of the ring of $d$-adic numbers~\cite{bartholdi_s:bsolitar,ahmed_s:polynomial_ergodicity} and study the automorphisms induced by polynomials. Here, we will use two other interpretations of the boundary $X^\infty$ as the ring $\Z_d[[t]]$ of formal power series over $\Z_d$, and as an infinite dimensional free $\Z_d$-module $\Z_d^\infty$ (which is a vector space over $\Z_d$ in the case of prime $d$). Both of these interpretations were studied in~\cite{savchuk_s:affine_automata}.

Each infinite word $a_0a_1a_2\ldots\in X^\infty$ can be represented as an element $a_0 + a_1t+\cdots+a_it^i+\cdots$ of $\Z_d[[t]]$.
Let $f(t)=a_0+a_1t+a_2t^2+\ldots$ and $g(t)=b_0+b_1t+b_2t^2+\ldots$ be power series in $\Z_d[[t]]$ with $a_0$ being a unit in $\Z_d$ (so that $f(t)$ is a unit in $\Z_d[[t]]$). We can define an affine transformation $\tau_{f,g}$ of $\Z_d[[t]]$ by
\[\bigl(\tau_{f,g}(h)\bigr)(t)=g(t)+h(t)\cdot f(t).\]
It is shown in~\cite{savchuk_s:affine_automata} that this transformation induces an automorphism of $X^*$ under the above identification of $X^\infty$ with $\Z_d[[t]]$. With a slight abuse of notation we will also denoted it by $\tau_{f,g}$. We will call these automorphisms \emph{$\Z_d[[t]]$-affine automorphisms of $X^*$}.

For example, an automorphism of the form $\tau_{1,g(t)}$ is a spherically homogeneous automorphism of $X^*$ that acts on the $i$-th letter of an input word by $(0,1,\ldots,d-1)^{b_i}\in\Sym(X)$, where $b_i$ is the coefficient at $t^i$ in $g(t)$. In particular, the addition of $t^n$ in $\Z_d[[t]]$ induces $\sigma^{(n)}\in\SHAut(X^*)$, and thus the group of automorphisms induced by addition of all possible polynomials in $\Z_d[[t]]$ is exactly $\Delta$.

A more general class of affine automorphisms of $X^*$ is obtained by viewing $X^\infty$ as an infinite dimensional free $\Z_d$-module $\Z_d^\infty$, where we treat a word $a_0a_1a_2\ldots\in X^\infty$ as an infinite-dimensional row ``vector'' $[a_0,a_1,a_2,\ldots]\in\Z_d^\infty$.
The set $\{\be_i\}_{i\geq 1}$, where $\be_i=[0,0, \ldots ,0,1,0, \ldots]$ with $1$ at position $i$, serves as a natural basis for $\Z_d^\infty$.

Let $A$ be an infinite upper triangular matrix with entries from $\Z_d$ whose diagonal entries are units in $\Z_d$. We will denote the set of all such matrices by $U_{\infty}\Z_d$. Also let $\bb\in\Z_d^\infty$ be a row vector. We define the transformation $\pi_{A,\bb}\colon \Z_d^\infty\to\Z_d^\infty$ by
\[\pi_{A,\bb}(\bx)=\bb+\bx\cdot A.\]
which is always well-defined since $A$ is upper triangular. As shown in~\cite{savchuk_s:affine_automata} every such transformation induces an automorphism of the tree $X^*$ which will be also denoted by $\pi_{A,\bb}$ and called an \emph{affine automorphism of $X^*$}.

The set $\Aff(X^*)=\{\pi_{A,\bb}\in\Aut(X^*)\mid A\in U_\infty\Z_d,\ \bb\in\Z_d^\infty\}$ forms a group which is called the \emph{group of affine automorphisms of} $X^*$. Let $I$ denote the infinite identity matrix over $\Z_d$. Then the set $\Aff_{I}(X^*)=\{\pi_{I,\bb}, \bb\in\Z_d^\infty\}$ is a subgroup of $\Aff(X^*)$ which is called the group of \emph{affine shifts}. Indeed, $\Aff_{I}(X^*)$ is the topological closure of the group $\Delta$ described above. Moreover, $\Aff_{I}(X^*)$ is a subgroup of the group $\SHAut(X^*)$ of spherically homogeneous automorphisms of $X^*$. They coincide when $d=2$.

The following two results from~\cite{savchuk_s:affine_automata} will be important in the proof of the main theorem in Section~\ref{sec:main}.

\begin{proposition}[\cite{savchuk_s:affine_automata}]
\label{pro:tau}
Let $f(t)=\sum_{n=0}^{\infty}a_nt^n,g(t)=\sum_{n=0}^{\infty}b_nt^n\in\Z_d[[t]]$ be two power series with $a_0$ a unit in $\Z_d$. Then the $\Z_d[[t]]$-affine automorphism $\tau_{f,g}$ coincides with the affine automorphism $\pi_{A,\bb}$ for $\bb=[b_0,b_1,b_2,\ldots]$ and
$$A=\left[\begin{array}{ccccc}
a_{0}&a_{1}&a_{2}&a_3&\ldots\\
0&a_{0}&a_{1}&a_2&\ldots\\
0&0&a_{0}&a_1&\ldots\\
0&0&0&a_{0}&\ldots\\
\vdots&\vdots&\vdots&\vdots&\ddots
\end{array}\right]$$
\end{proposition}

Thus, the set $\Aff_{[[t]]}(X^*)=\{\tau_{f,g}\in\Aut(X^*)\mid f,g\in\Z_d[[t]], f(t)\ \text{ is a unit in }\ Z_d[[t]]\}$ forms a proper subgroup of $\Aff(X^*)$ and is called the \emph{group of $\Z_d[[t]]$-affine automorphisms of} $X^*$.

\begin{theorem}[\cite{savchuk_s:affine_automata}]
\label{thm:normalizer}
The normalizer $N$ of the group $\Aff_{I}(X^*)$ in $\Aut(X^*)$ coincides with the group $\Aff(X^*)$ of all affine automorphisms. In particular, in the case of the binary tree, the normalizer of $\SHAut(\{0,1\}^*)$ in $\Aut(\{0,1\}^*)$ is $\Aff(\{0,1\}^*)$.
\end{theorem}

\section{The Structure of Group $\G$}
\label{sec:main}
In this section we will study the structure of the group $\G$ generated by the $4$-state automaton $\A$ whose Moore diagram is shown in Figure~\ref{fig:aut}. We will show that $\G$ is isomorphic to the rank 2 \emph{lamplighter} group $\LL_{2,2}\cong\Z_2^2 \wr \Z$. For that, we use the technique similar to the one developed in~\cite{savchuk_s:affine_automata}. We also used a \verb"GAP" package \verb"Automgrp"~\cite{muntyan_s:automgrp} to perform and check most of the calculations here.
\begin{figure}[!h]
         \begin{center}
         \includegraphics{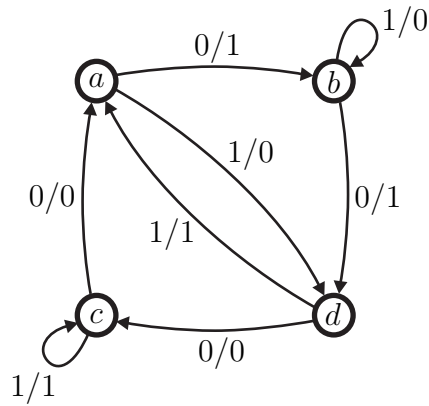}
         \end{center}
 \caption{The Automaton $\A$ generating the group~$\G$.\label{fig:aut}}
 \end{figure}

The wreath recursion of $\G$ is given by:
$$\begin{array}{lcl}
a&=&(b,d)\sigma,\\
b&=&(d,b)\sigma,\\
c&=&(a,c),\\
d&=&(c,a).
\end{array}$$
Let us put $x=ab^{-1}$, $y=ac^{-1}$ and $z=ad^{-1}$. It is straightforward to verify that the subgroup $\langle x,y,z\rangle$ is isomorphic to the $4$-element Klein group $\Z_2^2$. In particular, we have $y=xz=zx$ and $x^2=z^2=1$ (so $x^{-1}=x$ and $z^{-1}=z$). Therefore, $\G=\langle x,z,a\rangle$.

Observe, that the following relations hold in $\G$:

\begin{equation}
\label{eq:relations}
\begin{array}{l}
x=ab^{-1}=ba^{-1}=dc^{-1}=cd^{-1},\\
y=ac^{-1}=ca^{-1}=bd^{-1}=db^{-1},\\
z=ad^{-1}=da^{-1}=bc^{-1}=cb^{-1}.
\end{array}
\end{equation}

Since $x=ab^{-1}=(bd^{-1},db^{-1})$, $y=ac^{-1}=(bc^{-1},da^{-1})\sigma$, $z=ad^{-1}=(ba^{-1},dc^{-1})\sigma$, using relations~\eqref{eq:relations} and the notation introduced in the preliminary section, we obtain
\[x=y^{(1)},\quad y=z^{(1)}\sigma,\quad z=x^{(1)}\sigma.\]
So $x,y,z\in \SHAut(X^*)$.

\begin{lemma}
\label{lem:norm}
The automorphism $a$ lies in the normalizer of the group $\SHAut(X^*)$.
\end{lemma}

\begin{proof}
Since $\Delta=\langle \sigma^{(0)},\sigma^{(1)},\ldots\rangle$ is dense in $\SHAut(X^*)$ in the case of binary tree, it suffices to show that $(\sigma^{(n)})^a, (\sigma^{(n)})^{a^{-1}}\in \SHAut(X^*)$ for $n=0,1,\ldots$ Direct calculations give
\[\begin{array}{l}
(\sigma^{(0)})^a=(d^{-1},b^{-1})\sigma(1,1)\sigma(b,d)\sigma=(d^{-1}b,b^{-1}d)\sigma,\\ d^{-1}b=b^{-1}d,\\
d^{-1}b=(c^{-1}d,a^{-1}b)\sigma,\\
c^{-1}d=a^{-1}b,\\
c^{-1}d=(a^{-1}c,c^{-1}a).\\
\end{array}
\]
Therefore, using again relations~\eqref{eq:relations} we obtain $(\sigma^{(0)})^a\in \SHAut(X^*)$. It follows also by direct calculations that
\begin{equation}
\label{equ:iterations}
\begin{array}{l}
(\sigma^{(n+1)})^a=((\sigma^{(n)})^d,(\sigma^{(n)})^b),\\
(\sigma^{(n+1)})^b=((\sigma^{(n)})^b,(\sigma^{(n)})^d),\\
(\sigma^{(n+1)})^c=((\sigma^{(n)})^a,(\sigma^{(n)})^c),\\
(\sigma^{(n+1)})^d=((\sigma^{(n)})^c,(\sigma^{(n)})^a)\\
\end{array}
\end{equation}
for $n=0,1,\ldots $. We claim that $(\sigma^{(n)})^a=(\sigma^{(n)})^b=(\sigma^{(n)})^c=(\sigma^{(n)})^d$ for $n=0,1,\ldots $. This can be proved by induction on $n$ as follows. We have
\[\begin{array}{l}
(\sigma^{(0)})^a=(d^{-1},b^{-1})\sigma(1,1)\sigma(b,d)\sigma=(d^{-1}b,b^{-1}d)\sigma,\\
(\sigma^{(0)})^b=(b^{-1},d^{-1})\sigma(1,1)\sigma(d,b)\sigma=(b^{-1}d,d^{-1}b)\sigma,\\
(\sigma^{(0)})^c=(a^{-1},c^{-1})(1,1)\sigma(a,c)=(a^{-1}c,c^{-1}a)\sigma,\\
(\sigma^{(0)})^d=(c^{-1},a^{-1})(1,1)\sigma(c,a)=(c^{-1}a,a^{-1}c)\sigma.\\
\end{array}
\]
Hence $(\sigma^{(0)})^a=(\sigma^{(0)})^b=(\sigma^{(0)})^c=(\sigma^{(0)})^d$. Now equations~\eqref{equ:iterations} prove our claim. It follows immediately that $(\sigma^{(n)})^a\in \SHAut(X^*)$ for $n=1,2,\ldots$ Similarly we can show that $(\sigma^{(n)})^{a^{-1}}$ for $n=0,1,\ldots$
\end{proof}

Lemma~\ref{lem:norm} together with Theorem~\ref{thm:normalizer} give the following corollary.

\begin{corollary}
The automorphism $a$ lies in $\Aff(X^*)$ and is equal to $\pi_{A,\bb}$ for the matrix $A$ with the $i$-th row $\ba_i=[0^{i-1},1,(1,0)^\infty]$ (depicted in Figure~\ref{fig:matrix}) and $\bb=[(1,1,0,0)^\infty]$.
\end{corollary}

\begin{figure}[!h]
         \begin{center}
         \includegraphics[width=200pt]{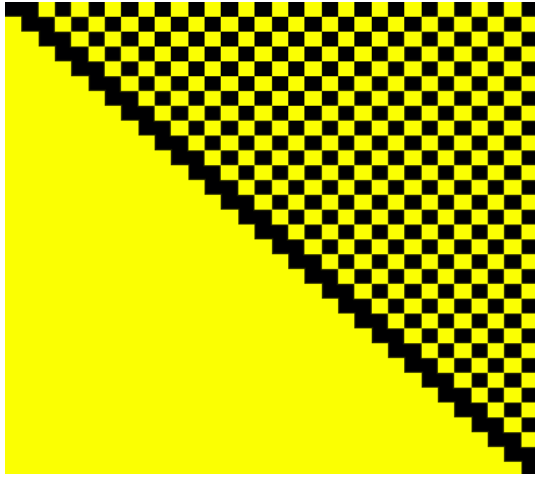}
         \end{center}
 \caption{A $32\times32$ minor of the matrix $A$ corresponding to the affine automorphism $a$.\label{fig:matrix}}
 \end{figure}

\begin{proof}
We can easily find $\bb$ by computing $\bb=\bb+[0,0,0,\ldots]\cdot A=\pi_{A,\bb}(0^\infty)=a(0^\infty)$. Let $\be_i=[0,0,\ldots,0,1,0,\ldots]$ be the $i$-th standard basis vector in $\Z_2^\infty$. Then we can compute the $i$-th row $\ba_i$ of matrix $A$ as follows. Since $a(\be_i)=\bb+\be_i\cdot A=\bb+\ba_i$, we obtain $\ba_i=a(\be_i)-\bb$. We leave the computations for the reader as an exercise.
\end{proof}

As immediate corollary, we obtain the following theorem.

\begin{theorem}
\label{thm:affine}
The generators $a,b,c$ and $d$ of $\G$ all lie in $\Aff_{[[t]]}(X^*)$ and are induced by the affine transformations of the form
\[
\begin{array}{lll}
a=\tau_{\frac{t^2+t+1}{t^2+1},\frac{1}{(t+1)^3}},&&b=\tau_{\frac{t^2+t+1}{t^2+1},\frac{t^2+t+1}{(t+1)^3}},\\ c=\tau_{\frac{t^2+t+1}{t^2+1},\frac{t}{(t+1)^3}},&&d=\tau_{\frac{t^2+t+1}{t^2+1},\frac{t^2}{(t+1)^3}}.
\end{array}
\]
\end{theorem}

\begin{proof}
By Proposition~\ref{pro:tau}, we obtain $a\in\Aff_{[[t]]}(X^*)$ with \[f=1+t+t^3+t^5+\cdots=1+t(1+t^2+t^4+\cdots)=1+\frac{t}{1+t^2}=\frac{t^2+t+1}{t^2+1},\]
and
\begin{multline*}
g=1+t+t^4+t^5+t^8+t^9+\cdots\\
=(1+t^4+t^8+\cdots)+t(1+t^4+t^8+\cdots)=\frac{1+t}{1+t^4}=\frac1{(t+1)^3},
\end{multline*}
where for simplification in the last step we used the fact that calculations are performed in $\Z_2$.
By Corollary 2.6 in~\cite{savchuk_s:affine_automata}, we can find the section of an affine automorphism $\tau_{f,g}$ at a vertex $x\in X$ via the formula
$$\tau_{f,g}|_x=\tau_{f,x\sigma(f)+\sigma(g)},$$
where $\sigma(c_0+c_1t+c_2t^2+\cdots)=c_1+c_2t+c_3t^2+\cdots$ for every formal power series $c_0+c_1t+c_2t^2+\cdots$. Using this formula, we can find the transformations inducing the automorphisms $b,c$ and $d$, where $b=a|_0$, $d=a|_1$ and $c=d|_0$.
\end{proof}

By Lemma~\ref{lem:norm}, conjugates of any spherically homogeneous element $s\in\SHAut(X^*)$ by powers (possibly negative) of $a$ are also spherically homogeneous, and hence commute. Therefore the following notation is well-defined for any $i_j\in\Z$:
$$s^{a^{i_1}+a^{i_2}+\cdots+a^{i_n}}:=s^{a^{i_1}}s^{a^{i_2}}\cdots s^{a^{i_n}}.$$
In particular, for each Laurent polynomial $p(a)\in\Z_2[a,a^{-1}]$ the elements $x^{p(a)}$ and $z^{p(a)}$ are defined. To show that $\langle x,z,a\rangle$ is isomorphic to $\Z_2^2\wr\Z$ it suffices to show that for each pair of Laurent polynomials $p(a)$ and $q(a)$ not both trivial, the element $x^{p(a)}z^{q(a)}$ of $\G$ is nontrivial. Actually, if it happens that for a pair of Laurent polynomials $p(a)$ and $q(a)$ we have $x^{p(a)}z^{q(a)}=1$, conjugating the last equation by a large enough power of $a$ gives $x^{\tilde{p}(a)}z^{\tilde{q}(a)}=1$ for some $\tilde{p},\tilde{q}\in \Z_2[a]$. Therefore, it is enough to show that for each pair of polynomials $p,q\in \Z_2[a]$ not both trivial $x^{p(a)}z^{q(a)}$ is a nontrivial element of $\G$. The proof of this fact will be based on the following lemmas.

\begin{lemma}
\label{lem:conj}
For any pair of polynomials $p,q\in \Z_2[a]$, we have
\[\begin{array}{l}
((x^{p(a)}z^{q(a)})^{(1)})^a=(x^{p(a)a}z^{q(a)a})^{(1)},\\
((x^{p(a)}z^{q(a)})^{(1)}\sigma)^a=(x^{p(a)a+a}z^{q(a)a+a})^{(1)}\sigma.\\
\end{array}
\]
\begin{proof}
We can write $a=(x^{-1},z^{-1})(a,a)\sigma=(x,z)(a,a)\sigma$ and $a^{-1}=(z^a,x^a)(a^{-1},a^{-1})\sigma$. Then using the fact that conjugates of $x$ and $z$ by powers of $a$ commute and that $x^2=z^2=1$, we obtain
$$((x^{p(a)}z^{q(a)})^{(1)})^a=(z^a,x^a)(a^{-1},a^{-1})\sigma(x^{p(a)}z^{q(a)})^{(1)}(x,z)(a,a)\sigma=$$$$(a^{-1}zx^{p(a)}z^{q(a)}za,a^{-1}xx^{p(a)}z^{q(a)}xa)=(a^{-1}x^{p(a)}z^{q(a)}a)^{(1)}=(x^{p(a)a}z^{q(a)a})^{(1)}$$ and
$$((x^{p(a)}z^{q(a)})^{(1)}\sigma)^a=(z^a,x^a)(a^{-1},a^{-1})\sigma(x^{p(a)}z^{q(a)})^{(1)}\sigma(x,z)(a,a)\sigma=$$$$(a^{-1}zx^{p(a)}z^{q(a)}xa,a^{-1}xx^{p(a)}z^{q(a)}za)\sigma=(a^{-1}x^{p(a)+1}z^{q(a)+1}a)^{(1)}\sigma=(x^{p(a)a+a}z^{q(a)a+a})^{(1)}\sigma$$
\end{proof}
\end{lemma}

\begin{lemma}
\label{lem:powers}
For each $n\geq0$, we have $x^{a^n}=(x^{a^n}z^{a^n})^{(1)}$ and
$$
z^{a^n} = \left\{
        \begin{array}{ll}
            (x)^{(1)}\sigma, & \quad n=0, \\
            (z^a)^{(1)}\sigma, & \quad n=1, \\
            (x^{a+a^2+\cdots+a^{n-1}}z^{a+a^2+\cdots+a^n})^{(1)}\sigma, & \quad n>1.
        \end{array}
    \right.
$$
\begin{proof}
For $n=0$ we have $x=x^{a^0}=(y)^{(1)}=(xz)^{(1)}$ and $z=z^{a^0}=(x)^{(1)}\sigma$. Using induction on $n$ from Lemma~\ref{lem:conj}, we immediately reach the statement of the lemma.
\end{proof}
\end{lemma}

Let us define
$$
\phi_n(a) = \left\{
        \begin{array}{ll}
            0 & \quad n=0 \\
            a+a^2+\cdots+a^n & \quad n>0
        \end{array}
    \right.
$$
For each polynomial $q(a)=\sum_{i=0}^nc_ia^i\in\Z_2[a]$ we define
\[\psi_q(a)=\sum_{i=1}^nc_i\phi_{i-1}(a).\]

\begin{lemma}
\label{lem:properties}
The functions $\phi_n$ and $\psi_q$ have the following properties:
\begin{itemize}
\item[(i)] For each $n\geq1$, we have $\phi_n(a)=a\phi_{n-1}(a)+a$.
\item[(ii)] If $\deg{q}\leq1$, then $\psi_q=0$.
\item[(iii)] If $\deg{q}\geq 2$, then $\deg{\psi_q}=\deg q-1$.
\item[(iv)] The function $\psi_q$ is linear in $q$.

\end{itemize}
\begin{proof}
The proof is straightforward and we leave it to the reader as an easy exercise.
\end{proof}
\end{lemma}

\begin{lemma}
\label{lem:pq}
For each pair of polynomials $p,q\in \Z_2[a]$, the section of $x^{p(a)}z^{q(a)}$ at each vertex of the first level is $x^{p(a)+\psi_q(a)+q(0)}z^{p(a)+a\psi_q(a)+aq(1)+aq(0)}$.
\begin{proof}
Assume $q(a)=\sum_{i=0}^nc_ia^i$. Then using the definition of the function $\psi_q$ together with Lemma~\ref{lem:powers}, we obtain the section of $z^{q(a)}$ at each vertex of the first level (say $z^{q(a)}|_0$) as
$$
z^{q(a)}|_0 = \left\{
        \begin{array}{ll}
            x^{c_1\phi_0(a)+c_2\phi_1(a)+\cdots+c_n\phi_{n-1}(a)}z^{c_1\phi_1(a)+c_2\phi_2(a)+\cdots+c_n\phi_n(a)}, & \quad c_0=0, \\
            x^{1+c_1\phi_0(a)+c_2\phi_1(a)+\cdots+c_n\phi_{n-1}(a)}z^{c_1\phi_1(a)+c_2\phi_2(a)+\cdots+c_n\phi_n(a)}, & \quad c_0=1
        \end{array}
    \right.
$$
$$
= \left\{
        \begin{array}{ll}
            x^{\psi_q(a)}z^{a\psi_q(a)+a(c_1+c_2+\cdots+c_n)}, & \quad c_0=0, \\
            x^{1+\psi_q(a)}z^{a\psi_q(a)+a(c_1+c_2+\cdots+c_n)}, & \quad c_0=1,
        \end{array}
    \right.
$$
$$=x^{q(0)+\psi_q(a)}z^{a\psi_q(a)+a(q(1)-q(0))}=x^{q(0)+\psi_q(a)}z^{a\psi_q(a)+aq(1)+aq(0)},$$
where we have used Lemma~\ref{lem:properties}(i) and the fact that $-1=1$ in $\Z_2$. It is obvious from Lemma~\ref{lem:powers} that $x^{p(a)}|_0=x^{p(a)}z^{p(a)}$. Now the statement of the lemma follows immediately from the fact that $x^{p(a)}z^{q(a)}$ is spherically homogeneous.
\end{proof}
\end{lemma}

\begin{remark}
\label{rem:odd-even}
Since $z$ acts nontrivially on the first level, it is clear that the action of $z^{q(a)}$ on the first level is trivial if and only if $q(a)$ has an even number of terms which, in the case of $\Z_2$, is equivalent to $q(1)=0$. In other words, $z^{q(a)}=(x^{q(0)+\psi_q(a)}z^{a\psi_q(a)+aq(0)})^{(1)}$ if $q(1)=0$ and $z^{q(a)}=(x^{q(0)+\psi_q(a)}z^{a\psi_q(a)+a+aq(0)})^{(1)}\sigma$ if $q(a)=1$.
\end{remark}

To simplify our notation, we will denote $x^{p(a)}z^{q(a)}$ by $(p,q)$ for each pair of polynomials $p,q\in\Z_2[a]$. To say that the section of $x^{p(a)}z^{q(a)}$ at each vertex of the first level is $x^{p'(a)}z^{q'(a)}$, we use the notation $(p,q)\rightarrow(p',q')$. According to Lemma~\ref{lem:pq}, we have $p'=p+\psi_q+q(0)$ and $q'=p+a\psi_q+aq(1)+aq(0)$.

Before we prove our main theorem, we need to introduce a remark and a lemma.

\begin{remark}
\label{rem:lower-deg}
The leading coefficient of any nonzero polynomial in $\Z_2[a]$ is $1$. So when we add two polynomials of the same degree, say $n$, the degree of the sum is less than $n$.
\end{remark}

\begin{lemma}
\label{lem:degree-sum}
Let $q(a)$ be a polynomial in $\Z_2[a]$ with degree at least $2$. Then $\deg{(\psi_q+\psi_{a\psi_q})}=\deg{q}-2$.
\begin{proof}
Let $q(a)=c_0+c_1a+\cdots+c_na^n$ with $n\geq2$ and $c_n\neq0$. Then
$$\psi_q(a)=a(c_2+c_3+\cdots+c_n)+a^2(c_3+c_4+\cdots+c_n)+\cdots+a^{n-2}(c_{n-1}+c_{n})+a^{n-1}c_n.$$
Hence
$$a\psi_q(a)=a^2(c_2+c_3+\cdots+c_n)+a^3(c_3+c_4+\cdots+c_n)+\cdots+a^{n-1}(c_{n-1}+c_{n})+a^nc_n.$$
So
$$\psi_{a\psi_q}(a)=a(c_2+2c_3+3c_4+\cdots+(n-1)c_n)+a^2(c_3+2c_4+3c_5+\cdots+(n-2)c_n)+\cdots$$$$+a^{n-2}(c_{n-1}+2c_n)+a^{n-1}c_n.$$
We finally obtain
$$\psi_q(a)+\psi_{a\psi_q}(a)=a(2c_2+3c_3+\cdots+nc_n)+a^2(2c_3+3c_4+\cdots+(n-1)c_n)+\cdots$$$$+a^{n-2}(2c_{n-1}+3c_n)+a^{n-1}(2c_n)$$$$=a(2c_2+3c_3+\cdots+nc_n)+a^2(2c_3+3c_4+\cdots+(n-1)c_n)+\cdots+a^{n-2}c_n.$$
Therefore, $\deg{(\psi_q+\psi_{a\psi_q})}=n-2=\deg{q}-2$.
\end{proof}
\end{lemma}

We are now ready to prove our main theorem.

\begin{theorem}
\label{thm:main}
The group $\G=\langle a=(b,d)\sigma,b=(d,b)\sigma,c=(a,c),d=(c,a)\rangle$ is isomorphic to the rank $2$ lamplighter group $\Z_2^2\wr\Z$.
\begin{proof}
From the paragraph preceding Lemma~\ref{lem:conj}, we only need to show that for each pair of polynomials $p,q\in\Z_2[a]$ not both trivial the expression $x^{p(a)}z^{q(a)}$ is nontrivial. We will assume the theorem is incorrect and prove it by contradiction as follows. We pick two polynomials $p,q\in\Z_2[a]$ such that $x^{p(a)}z^{q(a)}$ is trivial with $\max{\{\deg{p},\deg{q}\}}$ minimal. We will find a pair of polynomials $\tilde{p},\tilde{q}\in\Z_2[a]$ such that $x^{\tilde{p}}z^{\tilde{q}}$ is trivial and $\max{\{\deg{\tilde{p}},\deg{\tilde{q}}\}}<\max{\{\deg{p},\deg{q}\}}$. To find these two polynomials, we use the fact that all the states of the trivial automorphism are trivial and so is the product of any two of them.

Using the notation introduced above, we start with a pair $(p,q)$ corresponding to the trivial element of $\G$ with $\deg{p}=m$ and $\deg{q}=n$ such that $\max{\{m,n\}}$ is minimal. So $(p,q)\rightarrow(p',q')$ where
\[p'=p+\psi_q+q(0)\quad \text{and}\quad q'=p+a\psi_q+aq(0)\]
(Note that $q(1)=0$ by Remark~\ref{rem:odd-even}). We consider four cases of what can happen with the degrees of polynomials and analyze the dynamics that arises when we compute the sections of corresponding elements.

\noindent \textbf{Case I.} $m>n$.\\
Using Lemma~\ref{lem:properties}(iii), we obtain $\deg{p'}=\deg{q'}=m$, which constitutes Case~II.

\noindent \textbf{Case II.} $m=n$.\\
If $m=n\leq1$, then $\psi_q=0$. In such a case, we have only six values of the expression $x^{p(a)}z^{q(a)}$ to consider, namely $x^{a+1},x^a,x,x^{a+1}z^{a+1},x^az^{a+1}$ and $xz^{a+1}$, which are all nontrivial (keep in mind that $q(1)=0$). So we will assume in Case II that $m\geq2$ (and also in Case I since we have checked all possible occurrences). Hence by Lemma~\ref{lem:properties}(iii) and Remark~\ref{rem:lower-deg}, we have $\deg{p'}=m$ and $\deg{q'}<m$. Therefore we get back to Case~I (which does not yet finish the proof, of course).

\noindent \textbf{Case III.} $m=n-1$.\\
If $n=1$ and $m=0$, we have only two values of the expression $x^{p(a)}z^{q(a)}$ to consider, namely $z^{a+1}$ and $xz^{a+1}$ which are both nontrivial. So we will assume in Case III that $n\geq2$. Hence again by Lemma~\ref{lem:properties}(iii) and Remark~\ref{rem:lower-deg}, we have $\deg{p'}<n-1$ and $\deg{q'}=n$, thus bringing us to Case~IV.

\noindent \textbf{Case IV.} $m<n-1$.\\
Here we always have $n\geq2$. We obtain $\deg{p'}=n-1$ and $\deg{q'}=n$. Which again brings us to Case~III.

Let $(p'',q'')$ be the state of $(p,q)$ at any vertex of the second level (they are all equal), i.e. $(p,q)\rightarrow(p',q')\rightarrow(p'',q'')$. We claim that in case I and Case IV above, $(p+p'',q+q'')$ is trivial (this is immediate) with $\max{\{\deg{(p+p'')},\deg{(q+q'')}\}}<\max{\{\deg{p},\deg{q}\}}$ and the two polynomials $p+p''$ and $q+q''$ are not both trivial (we will show only that $p+p''$ is nontrivial). Since Case I leads to Case II in the first level and vice versa, and the same thing happens with Case III and Case IV, it is enough to consider only Case I and Case IV.

The polynomial $p''$ can be easily computed using Lemma~\ref{lem:properties}(ii),(iv) and it is equal to
\begin{multline*}p''=p'+\psi_{q'}+q'(0)=(p+\psi_q+q(0))+\psi_{p+a\psi_q+aq(0)}+p(0)\\
=(p+\psi_q+q(0))+\psi_p+\psi_{a\psi_q}+p(0)=p+\psi_q+\psi_p+\psi_{a\psi_q}+q(0)+p(0)
\end{multline*}
and thus
\[p+p''=\psi_q+\psi_p+\psi_{a\psi_q}+q(0)+p(0).\]

In Case I, we have $\deg{q},\deg{q''}<m$ so $\deg{(q+q'')}<m$. By Lemma~\ref{lem:properties}(iii), $\deg{(p+p'')}=m-1<m$. Since $m\geq2$, the polynomial $p+p''$ is nontrivial.

In Case IV, $\deg{q}=\deg{q''}=n$. Hence by Remark~\ref{rem:lower-deg}, $\deg{(q+q'')}<n$. By Lemma~\ref{lem:degree-sum}, $\deg{(p+p'')}=n-2<n$. For $n\geq3$, the polynomial $p+p''$ is nontrivial. We still have to check the case when $n=2$ and $m=0$ (keeping in mind that $q(1)=0$). There are four values of the expression $x^{p(a)}z^{q(a)}$ to consider, namely $x^{a^2},x^{a^2+a},x^{a^2+1},x^{a^2+a+1}$, which are all nontrivial. The proof is now complete.
\end{proof}
\end{theorem}

\bibliographystyle{alpha}
%\bibliography{References}

\begin{thebibliography}{GLS{\.Z}00}

\bibitem[AS17]{ahmed_s:polynomial_ergodicity}
Elsayed Ahmed and Dmytro Savchuk.
\newblock Endomorphisms of regular rooted trees induced by the action of
  polynomials on the ring $\mathbb {Z}_d$ of $d$-adic integers.
\newblock Preprint: arxiv:1711.06735, 2017.

\bibitem[Ati76]{atiyah:elliptic}
M.~F. Atiyah.
\newblock Elliptic operators, discrete groups and von {N}eumann algebras.
\newblock In {\em Colloque ``Analyse et Topologie'' en l'Honneur de Henri
  Cartan (Orsay, 1974)}, pages 43--72. Ast\'erisque, No. 32--33. Soc. Math.
  France, Paris, 1976.

\bibitem[Aus13]{austin:irrational_dimension13}
Tim Austin.
\newblock Rational group ring elements with kernels having irrational
  dimension.
\newblock {\em Proc. Lond. Math. Soc. (3)}, 107(6):1424--1448, 2013.

\bibitem[BDR16]{bondarenko_dr:lamplighterZ3wrZ16}
Ievgen Bondarenko, Daniele D'Angeli, and Emanuele Rodaro.
\newblock The lamplighter group {$\Bbb{Z}_3\wr\Bbb{Z}$} generated by a
  bireversible automaton.
\newblock {\em Comm. Algebra}, 44(12):5257--5268, 2016.

\bibitem[BM17]{bartholdi_m:order_undecidable}
Laurent Bartholdi and Ivan Mitrofanov.
\newblock The word and order problems for self-similar and automata groups.
\newblock Preprint: arxiv:1710.10109, 2017.

\bibitem[B{\v{S}}06]{bartholdi_s:bsolitar}
Laurent~I. Bartholdi and Zoran {\v{S}}uni{\'k}.
\newblock Some solvable automaton groups.
\newblock In {\em Topological and Asymptotic Aspects of Group Theory}, volume
  394 of {\em Contemp. Math.}, pages 11--29. Amer. Math. Soc., Providence, RI,
  2006.

\bibitem[BTZ17]{brieussel_tz:random_walks_affine_group}
J\'er\'emie Brieussel, Ryokichi Tanaka, and Tianyi Zheng.
\newblock Random walks on the discrete affine group.
\newblock Preprint: arxiv:1703.07741, 2017.

\bibitem[Cap14]{caponi:thesis2014}
Louis Caponi.
\newblock {On Classification of Groups Generated by Automata with 4 States over
  a 2-Letter Alphabet}.
\newblock Master's thesis, University of South Florida, Department of
  Mathematics and Statistics, Tampa, FL, 33620, USA, 2014.

\bibitem[DS02]{dicks_s:spectral_measure02}
Warren Dicks and Thomas Schick.
\newblock The spectral measure of certain elements of the complex group ring of
  a wreath product.
\newblock {\em Geom. Dedicata}, 93:121--137, 2002.

\bibitem[Gil18]{gillibert:order_undecidable18}
Pierre Gillibert.
\newblock An automaton group with undecidable order and {E}ngel problems.
\newblock {\em J. Algebra}, 497:363--392, 2018.

\bibitem[GK14]{grigorchuk_k:subgroup_structure14}
R.~Grigorchuk and R.~Kravchenko.
\newblock On the lattice of subgroups of the lamplighter group.
\newblock {\em Internat. J. Algebra Comput.}, 24(6):837--877, 2014.

\bibitem[GLS{\.Z}00]{grigorch_lsz:atiyah}
Rostislav~I. Grigorchuk, Peter Linnell, Thomas Schick, and Andrzej {\.Z}uk.
\newblock On a question of {A}tiyah.
\newblock {\em C. R. Acad. Sci. Paris S\'er. I Math.}, 331(9):663--668, 2000.

\bibitem[GM05]{gl_mo:compl}
Yair Glasner and Shahar Mozes.
\newblock Automata and square complexes.
\newblock {\em Geom. Dedicata}, 111:43--64, 2005.
\newblock (available at \emph{http://arxiv.org/abs/math.GR/0306259}).

\bibitem[GNS00]{gns00:automata}
R.~I. Grigorchuk, V.~V. Nekrashevich, and V.~I. Sushchanski{\u\i}.
\newblock Automata, dynamical systems, and groups.
\newblock {\em Tr. Mat. Inst. Steklova}, 231(Din. Sist., Avtom. i Beskon.
  Gruppy):134--214, 2000.

\bibitem[Gra14]{grabowski:atiyah_problem14}
{\L}ukasz Grabowski.
\newblock On {T}uring dynamical systems and the {A}tiyah problem.
\newblock {\em Invent. Math.}, 198(1):27--69, 2014.

\bibitem[Gra16]{grabowski:irrational_l2_lamplighter16}
{\L}ukasz Grabowski.
\newblock Irrational {$l^2$} invariants arising from the lamplighter group.
\newblock {\em Groups Geom. Dyn.}, 10(2):795--817, 2016.

\bibitem[GS14]{grigorch_s:essfree}
Rostislav Grigorchuk and Dmytro Savchuk.
\newblock Self-similar groups acting essentially freely on the boundary of the
  binary rooted tree.
\newblock In {\em Group theory, combinatorics, and computing}, volume 611 of
  {\em Contemp. Math.}, pages 9--48. Amer. Math. Soc., Providence, RI, 2014.

\bibitem[G{\.Z}01]{grigorch_z:lamplighter}
Rostislav~I. Grigorchuk and Andrzej {\.Z}uk.
\newblock The lamplighter group as a group generated by a 2-state automaton,
  and its spectrum.
\newblock {\em Geom. Dedicata}, 87(1-3):209--244, 2001.

\bibitem[KmV83]{kaimanovich_v:random_walks83}
V.~A. Ka\u\i~manovich and A.~M. Vershik.
\newblock Random walks on discrete groups: boundary and entropy.
\newblock {\em Ann. Probab.}, 11(3):457--490, 1983.

\bibitem[KPS16]{klimann_ps:orbit_automata}
Ines Klimann, Matthieu Picantin, and Dmytro Savchuk.
\newblock Orbit automata as a new tool to attack the order problem in automaton
  groups.
\newblock {\em J. Algebra}, 445:433--457, 2016.

\bibitem[LPP96]{lyons_pp:random_walks_lamplighter96}
Russell Lyons, Robin Pemantle, and Yuval Peres.
\newblock Random walks on the lamplighter group.
\newblock {\em Ann. Probab.}, 24(4):1993--2006, 1996.

\bibitem[LW13]{lehner_w:lamplighter_atiyah13}
Franz Lehner and Stephan Wagner.
\newblock Free lamplighter groups and a question of {A}tiyah.
\newblock {\em Amer. J. Math.}, 135(3):835--849, 2013.

\bibitem[MS16]{muntyan_s:automgrp}
Y.~Muntyan and D.~Savchuk.
\newblock {\em {AutomGrp -- \verb+GAP+ package for computations in self-similar
  groups and semigroups, Version 1.3}}, 2016.
\newblock {A}ccepted \verb+GAP+ package (available at
  \emph{http://www.gap-system.org/Packages/automgrp.html}).

\bibitem[Nek05]{nekrash:self-similar}
Volodymyr Nekrashevych.
\newblock {\em Self-similar groups}, volume 117 of {\em Mathematical Surveys
  and Monographs}.
\newblock American Mathematical Society, Providence, RI, 2005.

\bibitem[NP11]{nekrashevych_p:scale_invariant}
Volodymyr Nekrashevych and G{\'a}bor Pete.
\newblock Scale-invariant groups.
\newblock {\em Groups Geom. Dyn.}, 5(1):139--167, 2011.

\bibitem[NS04]{nekrash_s:12endomorph}
V.~Nekrashevych and S.~Sidki.
\newblock Automorphisms of the binary tree: state-closed subgroups and dynamics
  of $1/2$-endomorphisms.
\newblock volume 311 of {\em London Math. Soc. Lect. Note Ser.}, pages
  375--404. {Cambridge Univ. Press}, 2004.

\bibitem[OS04]{olijnyk_s:free_prods_unitriangular04}
Andrij Olijnyk and Vitaly Sushchansky.
\newblock Representations of free products by infinite unitriangular matrices
  over finite fields.
\newblock {\em Internat. J. Algebra Comput.}, 14(5-6):741--749, 2004.
\newblock International Conference on Semigroups and Groups in honor of the
  65th birthday of Prof. John Rhodes.

\bibitem[PSC02]{pittet_s:random_walks02}
C.~Pittet and L.~Saloff-Coste.
\newblock On random walks on wreath products.
\newblock {\em Ann. Probab.}, 30(2):948--977, 2002.

\bibitem[SS05]{silva_s:lamplighter05}
P.~V. Silva and B.~Steinberg.
\newblock On a class of automata groups generalizing lamplighter groups.
\newblock {\em Internat. J. Algebra Comput.}, 15(5-6):1213--1234, 2005.

\bibitem[SS16]{savchuk_s:affine_automata}
Dmytro~M. Savchuk and Said~N. Sidki.
\newblock Affine automorphisms of rooted trees.
\newblock {\em Geom. Dedicata}, 183:195--213, 2016.

\bibitem[SV11]{savchuk_v:free_prods}
Dmytro Savchuk and Yaroslav Vorobets.
\newblock Automata generating free products of groups of order 2.
\newblock {\em J. Algebra}, 336(1):53--66, 2011.

\bibitem[SVV11]{steinberg_vv:series_free}
Benjamin Steinberg, Mariya Vorobets, and Yaroslav Vorobets.
\newblock Automata over a binary alphabet generating free groups of even rank.
\newblock {\em Internat. J. Algebra Comput.}, 21(1-2):329--354, 2011.

\bibitem[VV07]{vorobets:aleshin}
Mariya Vorobets and Yaroslav Vorobets.
\newblock On a free group of transformations defined by an automaton.
\newblock {\em Geom. Dedicata}, 124:237--249, 2007.

\bibitem[VV10]{vorobets:series_free}
Mariya Vorobets and Yaroslav Vorobets.
\newblock On a series of finite automata defining free transformation groups.
\newblock {\em Groups Geom. Dyn.}, 4(2):377--405, 2010.

\end{thebibliography}

\def\cprime{$'$}

\end{document}